\documentclass[10pt,reqno]{amsart}
\usepackage[T1]{fontenc}
\usepackage{lmodern}
\usepackage{graphicx}
\usepackage{amsmath,amssymb}
\usepackage{microtype}
\usepackage[colorlinks=true,linkcolor=blue,citecolor=red,urlcolor=blue]{hyperref}
\hypersetup{
  pdftitle={A Numerical Method for the Fractional Semilinear Calder\'on Problem},
  pdfauthor={Mukul Dwivedi, Jesse Railo, Andreas Rupp},
  pdfsubject={Numerical reconstruction for a fractional semilinear inverse problem},
  pdfkeywords={Fractional inverse problem, finite element method, single measurement}
}
\numberwithin{equation}{section}
\newtheorem{theorem}{Theorem}[section]
\newcommand{\DRRnorm}[1]{\lVert#1\rVert}

\title[The fractional semilinear Calder\'on problem]
{A Numerical Method for the Fractional Semilinear Calder\'on Problem}
\author[M. Dwivedi]{Mukul Dwivedi}
\address{Department of Mathematics, Saarland University, Saarbr\"ucken, Germany}
\email{mukul.dwivedi@uni-saarland.de}
\thanks{Corresponding author: Mukul Dwivedi.}
\author[J. Railo]{Jesse Railo}
\address{Computational Engineering, School of Engineering Sciences, Lappeenranta-Lahti University of Technology, Finland}
\email{jesse.railo@lut.fi}
\author[A. Rupp]{Andreas Rupp}
\address{Department of Mathematics, Saarland University, Saarbr\"ucken, Germany}
\email{andreas.rupp@uni-saarland.de}
\date{}

\begin{document}

\begin{abstract}
We present a finite element method for reconstructing a coefficient in a fractional semilinear equation from a single exterior measurement.
\end{abstract}
\maketitle
\raggedbottom

\section{Introduction}\label{sec:intro}

We consider the inverse problem of determining the coefficient $q$ in
\begin{equation}\label{eq:model}
(-\Delta)^s u+qu^p=0\quad\text{in }\Omega,\qquad
u=f\quad\text{in }\Omega_e:=\mathbb R^d\setminus\overline\Omega.
\end{equation}
The fractional order $0<s<1$ and the integer $p\ge2$ are known. We prescribe one exterior datum $f$, supported in an open set $W_1\subset\Omega_e$. We measure $g:=((-\Delta)^su)|_{W_2}$ on another open set $W_2\subset\Omega_e$. Our aim is to reconstruct $q$ on a subdomain $D\Subset\Omega$ from this one pair $(f,g)$.

Unique continuation for the fractional Laplacian \cite[Thm. 1.2]{DRR_GSU} gives single-measurement reconstruction for the linear problem \cite[Thm. 1]{DRR_GRSU}. For semilinear equations, Lai--Lin \cite[Thm. 5.1]{DRR_LL} identify the nonlinear term at the solution values produced by one small exterior datum. Lipschitz stability on finite-dimensional coefficient classes is proved in \cite[Thm. 1.3]{DRR_Lin} using higher derivatives at zero of the map from exterior data to measurements. Here $q,f\ge0$ give a unique nonnegative solution without a smallness assumption on $f$.

We extend the Galerkin--Tikhonov method of \cite{DRR_potential} to the semilinear perturbation $qu^p$. The forward problem is nonlinear. The equation used to recover $u$ from exterior data is still linear. The power $u^p$ changes the propagation of errors and the determination of $q$ where $u$ is small. Our estimates quantify both of the latter effects.

\section{Setting and the existence of positive solutions}\label{sec:setting}

Let $\Omega\subset\mathbb R^d$ be a bounded Lipschitz domain, and let $D\Subset\Omega$ be a bounded Lipschitz subdomain. Let $W_1,W_2\Subset\Omega_e$ be nonempty bounded open sets. Write $X=\widetilde H^s(\Omega):=\overline{C_c^\infty(\Omega)}^{H^s(\mathbb R^d)}$. Define
\begin{equation}\label{eq:energy}
\begin{aligned}
a_s(v,z)&=\frac{C_{d,s}}2\iint_{\mathbb R^{2d}}
\frac{(v(x)-v(y))(z(x)-z(y))}{|x-y|^{d+2s}}\,dx\,dy,
\\
\DRRnorm v_X^2&=a_s(v,v)+\DRRnorm v_\Omega^2.
\end{aligned}
\end{equation}
Here $(\cdot,\cdot)_G$ and $\DRRnorm\cdot_G$ denote the $L^2(G)$ inner product and norm for an open set \(G\subset \mathbb R^d\), and we use $C_{d,s}=4^s\Gamma(d/2+s)/(\pi^{d/2}|\Gamma(-s)|)$. Write $(v,z)_X=a_s(v,z)+(v,z)_\Omega$. Assume $q\in L^\infty(\Omega)$, $0\le q\le Q<\infty$, and $0\ne f\in C_c^\infty(W_1)$, with $f\ge0$. Extend $f$ by zero and set $M=\|f\|_{L^\infty(\mathbb R^d)}$.

\textit{Existence and bounds.} Write $t_+=\max\{t,0\}$ and $t_-=\max\{-t,0\}$. Set $T(t)=\min\{M,t_+\}$ and $\Phi(t)=\int_0^tT(r)^p\,dr$. Minimize the functional 
\begin{equation}\label{eqn:functional}
\tfrac12a_s(f+v,f+v)+\int_\Omega q\Phi(v)\qquad  \text{ over }\quad  v\in X.     
\end{equation}
The fractional Poincar\'e inequality $\|z\|_{L^2(\Omega)}^2\le C a_s(z,z)$ for $z\in X$ implies that \eqref{eqn:functional} is coercive on $X$. The quadratic energy is strictly convex, and the nonlinear term is convex because $q\ge0$ and $\Phi$ is convex. Both terms are weakly lower semicontinuous. The direct method of the calculus of variations therefore gives a unique minimizer $v\in X$. Then $u=f+v$ satisfies $a_s(u,z)+(qT(u)^p,z)_\Omega=0$ for $z\in X$, since $f=0$ in $\Omega$. Both $u_-$ and $(u-M)_+$ belong to $X$, since $0\le f\le M$ outside $\Omega$. Since $T(u)^pu_-=0$, the first test gives $0=a_s(u,u_-)\le-a_s(u_-,u_-)$. Thus $u\ge0$. For $z=(u-M)_+$, both terms are nonnegative and $a_s(u,z)\ge a_s(z,z)$, so $z=0$. Hence $0\le u\le M$, $T(u)=u$, and $u$ solves \eqref{eq:model}. Subtracting the equations for two nonnegative solutions and testing with their difference proves uniqueness by the monotonicity of $t^p$ on $[0,\infty)$.

\textit{Positivity of the solution $u$ on $D$.} In fact, $u\ge m_D>0$ almost everywhere on $D$. To see this, write $v=u-f$. Then $(-\Delta)^sv+qv^p=h_f$ in $\Omega$, where $h_f=-(-\Delta)^sf$. The kernel formula gives $h_0:=\inf_\Omega h_f>0$. Choose $\psi\in C_c^\infty(\Omega)$ with $0\le\psi\le1$ and $\psi=1$ on $D$. Choose $t>0$ so that $t\|(-\Delta)^s\psi\|_\infty+Qt^p\le h_0$. Then $t\psi$ is a subsolution of the equation for \(v\). Testing the difference of the two inequalities with $(t\psi-v)_+$ gives $v\ge t\psi$. We may therefore take $m_D=t$. Unique continuation \cite[Thm. 1.2]{DRR_GSU} shows that two solutions with the same exterior data and measurement coincide. Their positivity on $D$ then identifies $q$ there through \eqref{eq:model}.

\textit{A linear equation for the unknown solution.} Let $Y=L^2(W_2)$. For $z\in X$, define
\begin{equation}\label{eq:continuation}
\mathcal Lz(x)=-C_{d,s}\int_\Omega
\frac{z(y)}{|x-y|^{d+2s}}\,dy\quad(x\in W_2),
\qquad
\mathcal Lz=((-\Delta)^sz)|_{W_2}.
\end{equation}
The sets $W_2$ and $\Omega$ are separated. Thus the kernel is square integrable on $W_2\times\Omega$, and $\mathcal L:X\to Y$ is compact. It is injective by \cite[Thm. 1.2]{DRR_GSU}.

First compute $u_0\in f+X$ satisfying $a_s(u_0,z)=0$ for all $z\in X$. This is \eqref{eq:model} with $q=0$ and depends only on $f,s,\Omega$. Set $r=u-u_0\in X$ and $g_0=((-\Delta)^su_0)|_{W_2}$. Linearity gives \( \mathcal Lr=g-g_0=:\mu\).
We recover $r$ by regularizing this linear observation equation, and then set $u=u_0+r$. 

\section{Three reconstruction steps}\label{sec:scheme}

Let $V_{h,0}\subset X$ be finite-dimensional spaces with $\inf_{z_h\in V_{h,0}}\|z-z_h\|_X\to0$ as \(h\to 0\) for every $z\in X$. On fitted meshes, continuous piecewise affine functions with zero exterior values give one choice for $0<s<1$; see \cite[Lem. 3.6]{DRR_potential}. We use conforming approximation spaces for $X$ and piecewise constants for the coefficient. Let $Y_\ell$ consist of constants on the cells of an observation partition $\mathcal T_\ell^W$. Its $L^2(W_2)$-orthogonal projection is $(\Pi_\ell y)|_J=|J|^{-1}\int_Jy$. Assume $\|\Pi_\ell y-y\|_Y\to0$ for every $y\in Y$. Related sampled observations are treated in \cite[Sec. 3.1]{DRR_potential}. The limits correspond to interior refinement $h\to0$ and observation refinement $\ell\to0$.

After subtracting the approximation of $g_0$, the measured data $\mu_\ell^\delta\in Y_\ell$ satisfy
\begin{equation}\label{eq:dataerror}
\|\mu_\ell^\delta-\Pi_\ell\mu\|_Y
\le\varepsilon_\delta:=\delta G_\ell+\eta_{\rm data},
\qquad
G_\ell:=\|\Pi_\ell g\|_Y.
\end{equation}
Here $\delta$ is relative noise, $\varepsilon_\delta$ is an absolute error bound, and $\eta_{\rm data}$ accounts for numerical measurement and subtraction errors. Initially, use exact $u_0$ and exact integrals.

\textit{Step 1: approximate $u$.} For $\alpha>0$, set $U_h=u_0+r_h$, where
\begin{equation}\label{eq:state}
r_h=\mathop{\rm argmin}_{z_h\in V_{h,0}}
\{\DRRnorm{\Pi_\ell\mathcal Lz_h-\mu_\ell^\delta}_Y^2
+\alpha\DRRnorm{z_h}_X^2\}.
\end{equation}
Here $\alpha$ is a stabilization parameter. Since $\alpha>0$, the functional is coercive and strictly convex on $V_{h,0}$, so $r_h$ exists and is unique.

\textit{Step 2: approximate $(-\Delta)^su$.} Set $w:=((-\Delta)^su)|_\Omega=-qu^p\in L^2(\Omega)$. Convergence $\|U_h-u\|_X\to0$ does not by itself ensure that $(-\Delta)^sU_h$ converges to $w$ in $L^2(\Omega)$. We approximate $w$ using the weak identity $a_s(u,z)=(w,z)_\Omega$ for $z\in X$. For $\tau>0$, define $w_h^\tau\in V_{h,0}$ by
\begin{equation}\label{eq:post}
(w_h^\tau,z_h)_\Omega+\tau a_s(w_h^\tau,z_h)
=a_s(U_h,z_h)\quad(z_h\in V_{h,0}).
\end{equation}
Since $U_h=u_0+r_h$ and $a_s(u_0,z_h)=0$, the right-hand side of \eqref{eq:post} equals $a_s(r_h,z_h)$.

\textit{Step 3: approximate $q$.} Set $\widehat U_h=\min\{M,\max\{-M,U_h\}\}$. This enforces $|\widehat U_h|\le M$ and gives $|\widehat U_h-u|\le|U_h-u|$. Let $Q_h$ be the space of constants on each cell of a partition $\mathcal T_h^D$ of $D$. For $\beta>0$, define
\begin{equation}\label{eq:coefficient}
q_h=\mathop{\rm argmin}_{t_h\in Q_h}
\{\DRRnorm{w_h^\tau+\widehat U_h^pt_h}_D^2
+\beta\DRRnorm{t_h}_D^2\}.
\end{equation}
Since $\beta>0$, the functional in \eqref{eq:coefficient} is coercive and strictly convex on $Q_h$. Its unique minimizer $q_h$ is characterized by setting the derivative in every direction $t_h\in Q_h$ equal to zero, that is,
\begin{equation}\label{eq:normal}
(\widehat U_h^pq_h,\widehat U_h^pt_h)_D+\beta(q_h,t_h)_D
=-(w_h^\tau,\widehat U_h^pt_h)_D\quad(t_h\in Q_h).
\end{equation}
Taking $t_h=\mathbf{1}_K$ in \eqref{eq:normal}, where $\mathbf{1}_K$ is the indicator of a coefficient cell $K$, gives
\begin{equation}\label{eq:cell}
q_h|_K=-\frac{\int_Kw_h^\tau\widehat U_h^p}
{\int_K|\widehat U_h|^{2p}+\beta|K|},
\qquad
c_{h,p}:=\min_{K\in\mathcal T_h^D}\frac{1}{|K|}\int_K|\widehat U_h|^{2p}.
\end{equation}
Thus $\|\widehat U_h^pt_h\|_D^2\ge c_{h,p}\|t_h\|_D^2$ for all $t_h\in Q_h$. Small values of $\widehat U_h$ make the coefficient poorly determined on that cell. The constant $c_{h,p}$ plays the same role as the visibility constant in \cite[Eq. (4.36) and Thm. 4.6]{DRR_conductivity}. 

\section{Error estimates}\label{sec:analysis}

Let $r_\alpha\in X$ minimize $\|\mathcal Lz-\mu\|_Y^2+\alpha\|z\|_X^2$. Define
\[
b(\alpha)=\|r_\alpha-r\|_X,\qquad
\eta_h(\alpha)=\inf_{z_h\in V_{h,0}}\|z_h-r_\alpha\|_X,\qquad
\kappa_\ell=\|(I-\Pi_\ell)\mathcal L\|_{X\to Y},
\]
where $I$ is the identity on $Y$. These are the regularization error, the interior approximation error and the error due to observation projection. Write $L=\|\mathcal L\|_{X\to Y}$. Then
\begin{equation}\label{eq:stateerror}
\|U_h-u\|_X\le\mathcal E
:=b(\alpha)+\sqrt{1+L^2/\alpha}\,\eta_h(\alpha)
+\frac{\kappa_\ell\|r\|_X}{\sqrt\alpha}
+\frac{\delta G_\ell+\eta_{\rm data}}{2\sqrt\alpha}.
\end{equation}
The proof follows the error splitting in \cite[Thm. 3.11]{DRR_potential}, with $b(\alpha)$ retained without a source condition. Projection in $\|z\|_{\alpha,\ell}^2=\|\Pi_\ell\mathcal Lz\|_Y^2+\alpha\|z\|_X^2$ gives the interior term. Comparing the optimality equations for $r_\alpha$ and the exact-data version of \eqref{eq:state}, using the preceding projection and $\|\mu-\mathcal Lr_\alpha\|_Y\le\sqrt\alpha\|r\|_X$, gives the term $\kappa_\ell\|r\|_X/\sqrt{\alpha}$. The noisy and exact-data normal equations of \eqref{eq:state}, with \eqref{eq:dataerror} and Young's inequality, give $\varepsilon_\delta/(2\sqrt\alpha)$. Injectivity gives $b(\alpha)\to0$ as $\alpha\to0$. Since $\mathcal L:X\to Y$ is compact and the orthogonal projections satisfy $\|\Pi_\ell y-y\|_Y\to0$ for every $y\in Y$, this convergence is uniform on the image under $\mathcal L$ of the unit ball of $X$. Hence $\kappa_\ell=\|(I-\Pi_\ell)\mathcal L\|_{X\to Y}\to0$.

\begin{theorem}\label{thm:error}
 Let $d\ge1$, $0<s<1$ and $p\ge2$ be an integer. Let $\Omega$ and $D\Subset\Omega$ be bounded Lipschitz domains. Let $W_1,W_2\Subset\Omega_e$ be nonempty bounded open sets. Assume $q\in L^\infty(\Omega)$, $0\le q\le Q<\infty$, and $0\ne f\in C_c^\infty(W_1)$ with $f\ge0$. Fix $\alpha,\tau,\beta>0$. Let $r_h$, $w_h^\tau$ and $q_h$ solve \eqref{eq:state}, \eqref{eq:post} and \eqref{eq:normal}, respectively, using exact $u_0$ and exact integration. Let $P_h:L^2(D)\to Q_h$ be the orthogonal projection. Define $\varepsilon_h=\|P_hq-q\|_D$, $E_u=\|\widehat U_h-u\|_D$, $E_h=a_s(U_h-u,U_h-u)^{1/2}$ and $E_w=\|w_h^\tau-w\|_\Omega$. With $\mathcal E$ defined in \eqref{eq:stateerror}, we have $E_u\le\mathcal E$ and $E_h\le\mathcal E$, and
\begin{align}
E_w\le\frac{E_h}{2\sqrt\tau}+A_h(\tau),\qquad
A_h(\tau):=\inf_{z_h\in V_{h,0}}
\big(\|z_h-w\|_\Omega^2+\tau a_s(z_h,z_h)\big)^{1/2},
\label{eq:werror}
\end{align}
\begin{align}
\|q_h-q\|_D\le\varepsilon_h+
\frac{E_w+pM^{p-1}\|q\|_{L^\infty(D)}E_u+M^p\varepsilon_h}
{\sqrt{c_{h,p}+\beta}}
+\frac{\beta\|P_hq\|_D}{c_{h,p}+\beta}.
\label{eq:qerror}
\end{align}
\end{theorem}

\begin{proof}
The bounds $E_u\le\mathcal E$ and $E_h\le\mathcal E$ follow from \eqref{eq:stateerror} and the definition of $\widehat U_h$. Let $y_h\in V_{h,0}$ solve \eqref{eq:post} with $u$ in place of $U_h$. Since $a_s(u,z_h)=(w,z_h)_\Omega$, the function $y_h$ minimizes $\|z_h-w\|_\Omega^2+\tau a_s(z_h,z_h)$. Thus $\|y_h-w\|_\Omega\le A_h(\tau)$. For $d_h=w_h^\tau-y_h$, subtraction and testing with $d_h$ give
\[
\|d_h\|_\Omega^2+\tau a_s(d_h,d_h)
\le E_h a_s(d_h,d_h)^{1/2}.
\]
The maximum of $E_ht-\tau t^2$ over $t\ge0$ is $E_h^2/(4\tau)$. Hence $\|d_h\|_\Omega\le E_h/(2\sqrt\tau)$. The triangle inequality proves \eqref{eq:werror}.

For the coefficient estimate, write $T_ht=\widehat U_h^pt$ and $\|t\|_{\beta,h}^2=\|T_ht\|_D^2+\beta\|t\|_D^2$. Set $e=q_h-P_hq$. Equation \eqref{eq:normal} gives $\|e\|_{\beta,h}^2=-(R,T_he)_D-\beta(P_hq,e)_D$, where
\[
R=(w_h^\tau-w)+(\widehat U_h^p-u^p)q+\widehat U_h^p(P_hq-q).
\]
For $|a|,|b|\le M$, we have $|a^p-b^p|\le pM^{p-1}|a-b|$. Therefore $\|R\|_D\le E_w+pM^{p-1}\|q\|_{L^\infty(D)}E_u+M^p\varepsilon_h$. Also, $\|T_he\|_D\le\|e\|_{\beta,h}$ and $\sqrt{c_{h,p}+\beta}\|e\|_D\le\|e\|_{\beta,h}$ by \eqref{eq:cell}. If $e\ne0$, division gives $\|e\|_{\beta,h}\le\|R\|_D+\beta\|P_hq\|_D/\sqrt{c_{h,p}+\beta}$. Dividing once more and using $\|q_h-q\|_D\le\|e\|_D+\varepsilon_h$ proves \eqref{eq:qerror}. The case $e=0$ is immediate.
\end{proof}

\textit{Errors in numerical evaluation.} The preceding theorem uses the exact \(s\)-harmonic solution $u_0$. In the numerical experiments, we evaluate its Poisson formula by quadrature, see \cite{DRR_Bucur}, and prescribe $f$ outside $\Omega$. Denote this approximation by $\widetilde u_0\in f+X$. The computed solution is then $U_h=\widetilde u_0+r_h$. Define $\varepsilon_0=\|\widetilde u_0-u_0\|_X$, and include the numerical error in $g_0$ in $\eta_{\rm data}$. For $z_h\in V_{h,0}$ and $t_h\in Q_h$, set $\|z\|_\tau^2=\|z\|_\Omega^2+\tau a_s(z,z)$ and define the residual functionals and their dual norms by
\[
\begin{aligned}
R_s(z_h)&=(\Pi_\ell\mathcal Lr_h-\mu_\ell^\delta,\Pi_\ell\mathcal Lz_h)_Y+\alpha(r_h,z_h)_X,\\
R_w(z_h)&=(w_h^\tau,z_h)_\Omega+\tau a_s(w_h^\tau,z_h)-a_s(U_h,z_h),\\
R_q(t_h)&=(w_h^\tau+\widehat U_h^pq_h,\widehat U_h^pt_h)_D+\beta(q_h,t_h)_D,\\
 \end{aligned}
 \]
\[
\rho_s=\sup_{0\ne z_h\in V_{h,0}}\frac{|R_s(z_h)|}{\|z_h\|_{\alpha,\ell}},\quad
\rho_w=\sup_{0\ne z_h\in V_{h,0}}\frac{|R_w(z_h)|}{\|z_h\|_\tau},\quad
\rho_q=\sup_{0\ne t_h\in Q_h}\frac{|R_q(t_h)|}{\|t_h\|_D}.
\]
The residuals are defined with exact integrals, so they also account for errors in numerical integration and matrix assembly. If the right-hand side of \eqref{eq:post} is evaluated as $a_s(r_h,z_h)$, the error $-a_s(\widetilde u_0,z_h)$ is included in $R_w(z_h)$ and is measured by $\rho_w$. Set

\[
\mathcal B_*:=b(\alpha)+\sqrt{1+L^2/\alpha}\,\eta_h(\alpha)
+\frac{\kappa_\ell\|r\|_X+\eta_{\rm data}/2+\rho_s}{\sqrt\alpha}
+\varepsilon_0,
\]
 $\theta=c_{h,p}+\beta$, and $C_{\tau,p}=(2\sqrt\tau)^{-1}+pM^{p-1}\|q\|_{L^\infty(D)}$. Applying the testing arguments used for \eqref{eq:stateerror} and in the proof of Theorem \ref{thm:error} to the equations with residuals $R_s,R_w,R_q$ gives
\begin{equation}\label{eq:totalerror}
\begin{split}
\|q_h-q\|_D\le{}&
\left(1+\frac{M^p}{\sqrt\theta}\right)\varepsilon_h
+\frac{A_h(\tau)+\rho_w}{\sqrt\theta}
+\frac{C_{\tau,p}\mathcal B_*}{\sqrt\theta}\\
&+\frac{C_{\tau,p}\delta G_\ell}{2\sqrt{\alpha\theta}}
+\frac{\beta\|P_hq\|_D+\rho_q}{\theta}.
\end{split}
\end{equation}
The fourth term shows the effect of measurement noise. Equation \eqref{eq:totalerror} gives $q_h\to q$ in $L^2(D)$ under the sufficient conditions that for \( h,\ell,\alpha,\tau,\beta\to0 \)
\[
 \liminf c_{h,p}>0,\quad \varepsilon_h,\rho_w,\rho_q\to0, \quad
\mathcal B_*/\sqrt\tau\to0,\quad \delta G_\ell/\sqrt{\alpha\tau}\to0.
\]
The condition $\liminf c_{h,p}>0$ is an additional assumption.

A logarithmic estimate could be sought through a continuous coefficient comparison as in \cite[Thm. 3.16 and Cor. 3.17]{DRR_potential}, but this requires uniform regularity and forward bounds, a controlled PDE residual throughout $\Omega$, and a stability theorem covering the chosen exterior sets.

\section{Numerical experiments}\label{sec:numerics}

We take $s=0.35$, $p\in\{1,2,3\}$, $\Omega=B_1(0)\subset\mathbb R^d$ and $D=B_{0.8}(0)$, with $d=1,2$. Here $B_R(0)$ is the open ball of radius $R$. The case $p=1$ provides a linear comparison.  Define $\zeta(t)=\exp(1-(1-t^2)^{-1})$ for $|t|<1$ and $\zeta(t)=0$ otherwise. Set
\begin{equation}\label{eq:numdata}
q(x)=\zeta\!\left(\frac{|x|}{0.75}\right),\qquad
f(x)=16\zeta\!\left(\frac{|x|-1.6}{0.2}\right).
\end{equation}
Thus $q\in C_c^\infty(D)$ and $M=16$. We choose $W_1=\{1.35<|x|<1.85\}$ and measure on $W_2=\{1.05<|x|<1.3\}\cup\{2<|x|<3\}$.

The numerical spaces contain only functions of $|x|$. For $h=1/n$, let $K_j=\{x:(j-1)h<|x|<jh\}$, $j=1,\ldots,n$. We approximate $r$ and $w$ by constants on these cells and use the cells contained in $D$ for $q_h$. Since $s<1/2$, these cell indicators belong to $X$; approximation is understood in its radial subspace. We take $h\in\{0.04,0.02,0.01,0.005,0.0025\}$, where $h$ is the width in the radial variable $|x|$.

Synthetic data use a separate nonlinear solver using Newton-type iteration with 257 weighted Jacobi basis functions \cite[Thm. 3]{DRR_DKK}. We compute $\widetilde u_0$ by 128-point Gauss--Legendre quadrature on $(1.4,1.8)$ in the radial ball Poisson formula \cite[Thm. 2.10]{DRR_Bucur}, and set $\widetilde u_0=f$ outside $\Omega$. Kernel quadrature then gives the approximation of $g_0$. Observations use 120 and 80 radial cells on the two intervals. 

\begin{figure}[tb]
\centering
\includegraphics[width=\linewidth]{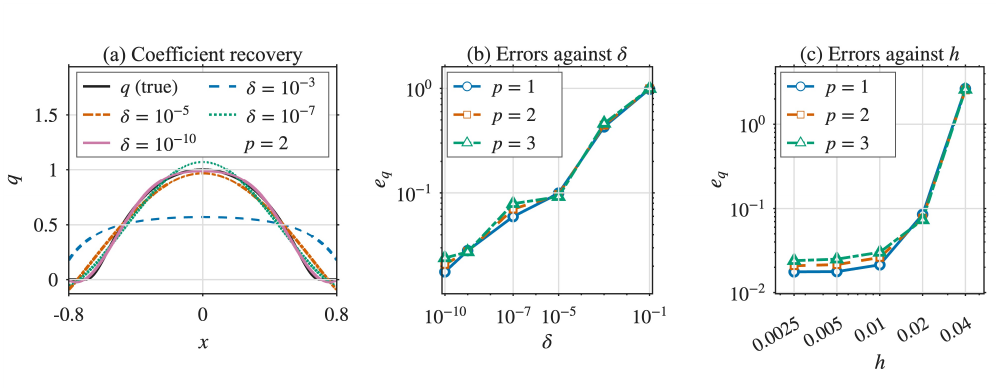}
\caption{(a): exact and reconstructed coefficients for $p=2$ at $h=0.0025$. (b): mean relative error $e_q$ against $\delta$ for $p=1,2,3$. (c): $e_q$ against $h$ at $\delta=10^{-10}$. }
\label{fig:ex1}
\end{figure}

For observation cells $J_i$, let $W=\operatorname{diag}(|J_i|)$ and $\|\mathbf z\|_W^2=\mathbf z^\top W\mathbf z$. Let $\mathbf g$ contain the cell averages of the full measurement $g$. We add noise before subtracting $g_0$:
\[
\mathbf g^\delta=\mathbf g+\delta\|\mathbf g\|_W
\frac{\boldsymbol\xi}{\|\boldsymbol\xi\|_W},\qquad
\delta\in\{10^{-1},10^{-3},10^{-5},10^{-7},10^{-9},10^{-10}\}.
\]
The entries of $\boldsymbol\xi$ are independent standard Gaussian variables, as in \cite[Sec. 4]{DRR_potential}. Five normalized vectors are reused across the three powers \(p\), noise levels and meshes. For realization $k$, set $e_q^{(k)}=\|q_h^{(k)}-q\|_D/\|q\|_D$. The plots show the mean relative error $e_q=\frac15\sum_{k=1}^5e_q^{(k)}$ and the coefficient from the first realization.

The noise study fixes $h=0.0025$ and $\tau=2.5\times10^{-7}$. For each data vector, we choose the largest candidate $\alpha=10^{-j/10}$, $j=0,\ldots,160$, for which $\|\Pi_\ell\mathcal Lr_h-\mu_\ell^\delta\|_Y\le1.05\delta\|\mathbf g\|_W$. We then set $\beta=10^{-10}|D|^{-1}\int_D|\widehat U_h|^{2p}$. The mesh study fixes $\delta=10^{-10}$. 

\noindent\textbf{Example 1: an interval.} We take $\Omega=(-1,1)$ and $D=(-0.8,0.8)$. Figure \ref{fig:ex1} shows improved reconstruction as noise decreases. Mesh and noise refinement reduce the mean error.

\noindent\textbf{Example 2: a disk.} We take $\Omega=B_1(0)\subset\mathbb R^2$. At $h=0.0025$, there are 400 unknown values for $r_h$, 320 coefficient values and 200 observations. The operator is two-dimensional, but only radial coefficients are reconstructed. Figure \ref{fig:ex2} shows decreasing mean error as noise decreases.

\begin{figure}[t]
\centering
\includegraphics[width=\linewidth]{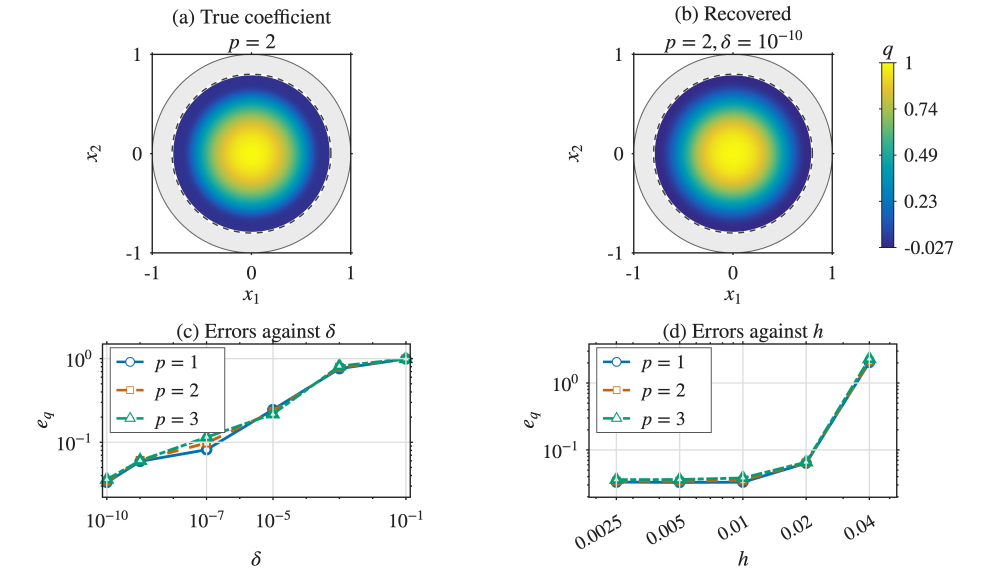}
\caption{(a) and (b): exact and reconstructed coefficients for $p=2$ at $\delta=10^{-10}$ and $h=0.0025$. Grey indicates $\Omega\setminus D$. (c) and (d): mean relative error $e_q$ against $\delta$ and $h$ for $p=1,2,3$. }
\label{fig:ex2}
\end{figure}

For each $p$, the mean coefficient error decreases as $\delta$ decreases, while the ordering of the errors for $p=1,2,3$ depends on the noise level. The errors have similar trends for each \(p=1,2,3\), although their values differ slightly.

\medskip
\noindent\textbf{Code and data availability.}
The MATLAB code, data and numerical supplement are available in the
\href{https://github.com/AndreasRupp/fractional-semilinear-calderon}{GitHub repository}.

\noindent\textbf{Acknowledgements.} Dwivedi and Rupp have been supported by the Deutsche Forschungsgemeinschaft (DFG, German Research Foundation) -- 577175348. Railo has been supported by the Research Council of Finland through the Flagship of Advanced Mathematics for Sensing, Imaging and Modelling (grant number 359183), the Emil Aaltonen Foundation, and the Jenny and Antti Wihuri Foundation.

\end{document}